\documentclass[11pt]{amsart}

\usepackage{amsmath,amsthm,amsfonts,amssymb,mathrsfs}
\usepackage{inputenc,mathrsfs}

\numberwithin{equation}{section}

\usepackage[dvips]{graphicx}
\usepackage[hidelinks]{hyperref}

\newtheorem{definition}{Definition}[section]
\newtheorem{theorem}{Theorem}[section]

\newtheorem{lemma}{Lemma}[section]

\theoremstyle{remark}
\newtheorem{remark}{Remark}[section]

\DeclareMathOperator{\diam}{\mathrm{diam}}
\DeclareMathOperator{\riem}{\mathrm{Rm}}
\DeclareMathOperator{\ric}{\mathrm{Ric}}
\DeclareMathOperator{\hess}{\mathrm{Hess}}

\DeclareMathOperator{\vol}{\mathrm{vol}}

\newcommand{\tm}{\textrm}
\newcommand{\rf}{\mathcal{RF}(n,B)}

\author{Panagiotis Gianniotis}

\address{Department of Mathematics, University of Toronto, 40 St George Street, Toronto, ON M5S 2E4, Canada}
 
\email{p.gianniotis@utoronto.ca}

\subjclass[2000]{}

\dedicatory{}


\begin{document}
\title[Regularity theory for Type I Ricci flows]{Regularity theory for Type I Ricci flows}
\maketitle
\begin{abstract} 
We consider Type I Ricci flows and obtain integral estimates for the curvature tensor valid up to, and including, the singular time. Our estimates  partially extend to higher dimensions a curvature estimate recently shown to hold in dimension three  by Kleiner and Lott in \cite{KleinerLott14}. To do this we adapt the technique of quantitative stratification, introduced by Cheeger--Naber in \cite{CheegerNaber13a}, to this setting.
\end{abstract}
\section{Introduction}

In this paper we study complete Ricci flows $(M,g(t))_{t\in [0,T)}$ satisfying a curvature bound of the form
\begin{equation}\label{introtype}
\sup_M |\riem(g(t))|_{g(t)}\leq \frac{B}{T-t},
\end{equation}
for all $t\in [0,T)$. If $(g(t))_{t\in [0,T)}$ becomes singular as $t\rightarrow T$, namely
\begin{equation}
\lim_{t\rightarrow T}\sup_M |\riem(g(t))|_{g(t)}=+\infty.
\end{equation}
the singularity is classified as Type I, hence we will refer to  \eqref{introtype} as a Type I curvature bound. This kind of singular behaviour for the Ricci flow is very common and it is in fact  conjectured that for closed manifolds $M$ such singularities are generic; see for instance \cite{AngenentIsenbergKnopf15,IKS17}.

Our results provide $L^p$ bounds for the curvature along the flow assuming  Type I bounds. For instance, we obtain the following theorem.

\begin{theorem}\label{thm:curv_estimates}
 Let $(M^n,g(t))_{t\in [0,T)}$, $\dim M=n$, be a \emph{compact} Ricci flow satisfying \eqref{introtype}. Then, for every non-negative integer $j$ and $p\in (0,2)$ there exist $C_{p,j}(g(0))<+\infty$ such that for every  $t\in [0,T]$
\begin{eqnarray}
\int_{M} |\nabla^j \riem(g(t))|_{g(t)}^{\frac{p}{j+2}} d\mu_{g(t)} &\leq& C_{p,j}, \label{eqn:introder_curv_estimate1}  \\
\textrm{and}\quad \int_0^T \int_{M} | \nabla^j \riem(g(s))  |_{g(s)}^{\frac{p+2}{j+2}}d\mu_{g(s)}ds &\leq& C_{p,j}. \label{eqn:introder_curv_estimate2}
\end{eqnarray}

If $(g(t))_{t\in[0,T)}$ becomes singular at $T$, estimate \eqref{eqn:introder_curv_estimate1} is valid  at $t=T$ on the set $\Omega=\{ x\in M,\; \sup_{t\in[0,T)}|\riem(g)|_{g}(x,t)<+\infty\}$. Moreover, if $g(t)$ has positive isotropic curvature and $n=4$, the estimates above hold for any $p\in(0,3)$.
\end{theorem}

Notice that estimate \eqref{eqn:introder_curv_estimate1} agrees with the recent curvature estimate obtained by Kleiner--Lott in \cite{KleinerLott14} in dimension three. Moreover, the results in \cite{KleinerLott14} hold without the Type I assumption and even after the first singularity occurs. Our results on the other hand are valid in any dimension, which may hint to a general fact about weak solutions to Ricci flow. Notions of weak solutions to Ricci flow have recently been proposed by Haslhofer--Naber in \cite{HaslhoferNaber15} as well as Sturm in \cite{Sturm16} and Kopfer--Sturm in \cite{KopferSturm16}.

In \cite{KleinerLott14} the curvature estimate is a consequence of the study of a certain class of space-time manifolds that arise naturally as limits of Perelman's Ricci flow with surgery, as the associated fineness parameter goes to zero. In contrast, our approach bypasses Ricci flow with surgery, and instead uses the tangent flow analysis and monotonicity formula available for Type I Ricci flows. In particular, we adapt the technique of \textit{quantitative stratification}, recently introduced by Cheeger--Naber in \cite{CheegerNaber13a}, to this setting.

The ideas in \cite{CheegerNaber13a} are very general and have been applied in a wide range of geometric PDE, leading to improved curvature estimates; see \cite{CheegerNaber13b, ChHasNab13, ChHasNab15, BreinerLamm14}. However, to adapt these ideas to the Ricci flow we need to overcome a few issues, which we describe below.

We may define the singular set $\Sigma$ of a Ricci flow as the set of points with no neighbourhood where the curvature remains bounded as $t\rightarrow T$.  Under assumption \eqref{introtype}, Naber shows in \cite{Naber10} that tangent flows at the singular time, namely limits of appropriate pointed sequences of rescalings, are gradient shrinking Ricci solitons. Previously \v{S}e\v{s}um \cite{Sesum06} had shown that this is true in the case of compact tangent flows. Then, Enders--M\"uller--Topping in \cite{EMT} show that tangent flows are non-flat if and only if they are `centered' around singular points. Mantegazza--M\"uller \cite{MantegazzaMuller15} also prove these facts using a different approach.

Imitating the classical regularity theory for minimal surfaces or harmonic maps, as developed for instance in \cite{Federer70,SchoenUhlenbeck82,Almgren83,Simon93,White97}, it is natural to consider the stratification $$\Sigma_0\subset \cdots \subset \Sigma_{n-1}=\Sigma$$
  of $\Sigma$, where
\begin{equation*}
\Sigma_k=\{x\in\Sigma, \textrm{no tangent flow at $x$ splits more than $k$ Euclidean factors}\}.
\end{equation*}
In fact $\Sigma=\Sigma_{n-2}$, since any shrinking soliton splitting more than $n-2$ Euclidean factors should be the Gaussian soliton in the Euclidean space.

A more detailed study of this stratification is done in \cite{Gianniotis2017}. There, a key issue is that the properties of each $\Sigma_k$  relevant to singularity formation, as captured by the amount of the Euclidean factors split by the tangent flows, do not interact with the geometric properties of each $\Sigma_k$ as a subset of $(M,g(t))$:
in the shrinking round sphere example, $\Sigma=\Sigma_0=\mathbb S^n$ is an $n$-dimensional subset, but it converges to a $0$-dimensional space towards the singular time.

This is in contrast to other situations, where the interest is in the geometry of the singular set as a subset of a given ambient space. Similar issues appear when we try to adapt the philosophy of \cite{CheegerNaber13a} in this paper. 

Below we describe the results of the paper in more detail:

In Section \ref{section:monotonicity} we recall a monotone quantity for possibly singular Type I Ricci flows and its associated density that was introduced in \cite{Gianniotis2017}, based on Perelman's reduced volume, extending ideas from \cite{CaoHamiltonIlmanen04,Enders08, Naber10}. This leads to the notion of the \textit{spine} of a shrinking Ricci soliton with bounded curvature: the set where the density function attains its \textit{minimum}. It is then shown that the spine satisfies a diameter estimate, modulo the splitting of Euclidean factors; see Theorem \ref{thm:spine_geometry}. In particular this estimate shows that, as the flow induced by the soliton appraches its singular time, the spine collapses to a Euclidean space. This is a key fact that allows us to adapt the ideas in \cite{CheegerNaber13a} to the setting of Type I Ricci flows.

Now, let  $\mathcal C(n,B,\kappa_0,\kappa_1)$ be the class of complete Ricci flows $(M,g(t))_{t\in(-2,0)}$, such that $\dim M=n$ and
\begin{enumerate}
\item $|\riem(g(-\tau))|_{g(-\tau)}\leq B/\tau$ in $M$, for every $\tau\in(0,2)$.
\item $g(t)$ is $\kappa_0$ non-collapsed below scale $ 1$, namely
$$\vol_{g(t)}( B_{g(t)}(x, r ) )\geq \kappa_0 r^n,$$
for every $(x,t)\in M\times (-2,0)$ and $r\leq 1$ for which $R(g(t))\leq r^{-2}$ in  $B_{g(t)}(x, r )$, $R$ denoting the scalar curvature.
\item $g(t)$ is $\kappa_1$ non-inflated below scale $1$, namely
$$\vol_{g(t)}( B_{g(t)}(x, r ) )\leq\kappa_1 r^n,$$
for every $(x,t)\in M\times [-1,0)$ and $r\leq 1$, $t-r^2>-2$, for which
$$R\leq \frac{c(n)B}{t-\bar t},$$
in $B_{g(t)}(x, r )$ for all $\bar t\in [t-r^2,t]$, where $c(n)<+\infty$ is a constant such that $|R(g)|\leq c(n)|\riem(g)|_g$, for any Riemannian metric $g$.
\end{enumerate}

In Section \ref{section:q_stratification},  following \cite{CheegerNaber13a}, we define the quantitative stratification $S^k_{\eta,\tau}$, where $k\geq 0$ is an integer, $\eta>0$ and $\tau\in (0,1]$, for each $(M,g(t))_{t\in (-2,0)}$ in $\mathcal{C}(n,B,\kappa_0,\kappa_1)$. The intuition behind the sets $S^k_{\eta,\tau}$ is that there is no scale $\bar \tau\in [\tau,1]$ at which the flow around $x\in S^k_{\eta,\tau}$  is $\eta$-close to a shrinking Ricci soliton that splits more than $k$ Euclidean factors. We refer the reader to Section \ref{section:q_stratification}  for the detailed definition.

The relationship of the sets $S^k_{\eta,\tau}$ to $\Sigma_k$ is given by
\begin{equation*}
\Sigma_k =  \bigcup_{\eta} \bigcap_{\tau}   S^k_{\eta,\tau}.
\end{equation*}
We show that the quantitative stratification satisfies the following volume estimate:

\begin{theorem} \label{thm:vol_estimate}
Let $(M,g(t))_{t\in (-2,0)} \in\mathcal C(n,B, \kappa_0,\kappa_1)$. Then, there exist $\alpha(B),\beta(B) \in (0,1)$ and $C_\eta=C(n,B,\kappa_0,\kappa_1,\eta) <+\infty$, such that for every $ 0<\tau\leq\alpha$ 
\begin{equation}\label{intro:vol_estimate}
\vol_{g(- \tau)}\left(S^k_{\eta, \tau}\cap B_{g(-\alpha)}(x,\beta)\right) \leq C_\eta   \tau^{\frac{n-k-\eta}{2}}.
\end{equation}
\end{theorem}

Then, in Section \ref{section:c_estimates}, we combine Theorem \ref{thm:vol_estimate} with the $\varepsilon$-regularity Lemmata \ref{lem:epsilon} and \ref{lem:epsilon_weyl}, to prove uniform curvature estimates for any Ricci flow $(M,g(t))_{t\in (-2,0)}$ in $\mathcal{C}(n,B,\kappa_0,\kappa_1)$.

Define the curvature radius of $(M,g(t))_{t\in (-2,0)}$ at $x\in M$ as
\begin{equation*}
r_{\riem}(x)=\sup\left\{r\leq 1, \; |\riem(g)|\leq r^{-2}\;\textrm{in}\; B_{g(-r^2)}(x,r)\times [-r^2,0] \right\}.
\end{equation*}
Note that if $(g(t))_{t\in (-2,0)}$ is singular at $x$, we define $r_{\riem}(x)=0$. 

Then, Theorem \ref{thm:curv_estimates} is a consequence of the following result.
 \begin{theorem}\label{introthm}
 Let $(M,g(t))_{t\in (-2,0)}\in \mathcal C(n,B,\kappa_0,\kappa_1)$. Then there exist $\alpha(B),\beta(B)>0$ such that for any integer $j\geq 0$ and any $p\in (0,2)$  there is $C_{p,j}= C_{p,j}(n,B,\kappa_0,\kappa_1)<+\infty$ such that
 \begin{align}
  \int_{B_{g(-\alpha)}(x,\beta)\cap\{r_{\riem}>0\}} | \nabla^j \riem(g(0)) |_{g(0)}^{\frac{p}{j+2}} d\mu_{g(0)} &\leq C_{p,j},\label{curvature1}
 \end{align}
Moreover, if $\dim M=4$ and $g(t)$ has positive isotropic curvature, then \eqref{curvature1} holds for any $p\in (0,3)$.
\end{theorem}

Observe that $\mathbb S^2 \times \mathbb R^{n-2}$ with the standard soliton structure satisfies the estimate of Theorem \ref{introthm} for $p=2$, so the theorem is not sharp. Similarly for the soliton $\mathbb S^3\times \mathbb R$, for $p=3$. On the other hand, if \eqref{curvature1} were to hold for $p=2$ in dimension three or $p=3$ in dimension four with positive isotropic curvature, this would imply quite strong control in the geometry of $(M,g(t))_{t\in [0,T)}$ in Theorem \ref{thm:curv_estimates}: by a result of Topping \cite{Topping05} the diameter of $(g(t))_{t\in [0,T)}$ would be uniformly bounded for all $t$; see also Zhang \cite{Zhang14}.

Finally, we note that Theorem \ref{introthm} is a consequence of stronger estimates on the curvature radius proven in Theorem \ref{thm:regularity1}; see also Theorem \ref{thm:integral_estimates}. Furthermore, the estimates of Theorems \ref{thm:curv_estimates} and \ref{introthm} can be strengthened to $p\in (0,n-1)$ under appropriate bounds on the Weyl curvature; see Remarks \ref{weyl1} and \ref{weyl2}. 

{\bf Acknowledgments:} The author would like to acknowledge support from the Fields Institute during the completion of this work, and thank University of Waterloo and Spiro Karigiannis for their hospitality. Moreover, the author is grateful to Robert Haslhofer for his interest in this work and for many discussions on an earlier draft of this paper.

\section{A monotonicity formula for singular Ricci flows.}\label{section:monotonicity}
In this section we describe a monotonicity formula, and its associated density, in the setting of a Ricci flow $(M,g(t))_{t\in (-T,0)}$, $T\in(0,+\infty]$ subject to a Type I curvature bound, namely  
\begin{equation} \label{rfnb}
\sup_{M}|\riem(g(t)|_{g(t)}\leq \frac{B}{|t|},
\end{equation}
for $t\in (-T,0)$, as  introduced in \cite{Gianniotis2017}. Note that we allow for the possibility that $$\lim_{t\rightarrow 0}\sup_M |\riem(g(t))|_{g(t)}=+\infty.$$

Let us introduce some notation we will use throughout the paper. Given a Ricci flow $(M,g(t))_{t\in (-T,0]}$, $T\in (0,+\infty]$, and $x\in M$, let $\mathfrak g$ denote the triplet $(M,g(t),x)_{t\in(-T,0) }$. When we want to distinguish between different pointed Ricci flows with the same underlying flow we will also use the notation $\mathfrak g_x$ to denote  $(M,g(t),x)_{t\in(-T,0)}$.

Moreover, for every $s>0$ we will denote the rescaled flow, pointed at $x$, by
$$(\mathfrak g_x)_s=(M,s^{-2}g(s^2t), x)_{t\in (-T,0)}.$$

\subsection{Perelman's reduced volume.} Let $(M,g(t))_{t\in [0,T]}$ be a complete smooth Ricci flow and let  $l_{(x,T)}$ denote the reduced distance function based at $(x,T)\in M\times (0,T]$, as introduced by Perelman in \cite{Perelman02}:
\begin{equation*}
l_{(x,T)}(y,\tau)=\inf\left\{\frac{1}{2\sqrt \tau} \int_0^{\tau} \sqrt{\bar\tau} \left( R(\gamma(\bar\tau), T-\bar\tau) +\left| \frac{d}{d\bar\tau} \gamma(\bar\tau)\right|_{g(T-\bar\tau)}^2  \right) d\bar\tau\right\},
\end{equation*}
where the infimum is taken over all curves $\gamma: [0, \tau] \rightarrow M$ with $\gamma(0)=x$ and $\gamma(\tau)=y$.

Then, as in \cite{Perelman02}, we may define the reduced volume at scale $\tau>0$  based at $(x,T)$:
\begin{equation}\label{eqn:red_vol}
\mathcal V_{(x,T)}(\tau)=\int_M \frac{e^{-l(y,\tau)}}{(4\pi \tau)^{n/2}} d\mu_{g(\tau)}(y).
\end{equation}
Perelman  discovered the remarkable fact that $\mathcal V_{(x,T)}(\tau)$ is monotone decreasing in $\tau$. Moreover, $\lim_{\tau\rightarrow 0} \mathcal V_{(x,T)}(\tau) =1$ and $\mathcal V_{(x,T)}(\tau)$ is constant if and only if $g(t)$ is the Euclidean space for every $t$.

With the notation introduced above, if $\mathfrak g=(M, g(t),x)_{t\in(-T,0)}$ we define
$$l_{\mathfrak g}(y, \tau)=l_{(x,0)}(y,\tau).$$

\subsection{The space of uniformly Type I flows.} Let $\mathcal{RF}(n,B)$ denote the collection of all complete pointed Ricci flows $(M,g(t),x)_{t\in (-T,0)}$, where $M$ is $n$-dimensional, $T\in (0,+\infty]$, and $g(t)$ satisfies \eqref{rfnb} for all $t\in (-T,0)$. 

Moreover, let $\mathcal{RF}_{reg}(n,B)$ be the collection of $(M,g(t),x)\in \mathcal{RF}(n,B) $ satisfying 
$$\sup_{M\times (-T,0)}|\riem(g(t))|_{g(t)}<+\infty.$$
Observe that any flow in $\mathfrak g=(M,g(t),x)_{t\in (-T,0)}\in\mathcal{RF}_{reg}(n,B)$ can be extended to a Ricci flow $(g(t))_{t\in (-T,0]}$ by Shi's estimates.

We endow $\mathcal{RF}(n,B)$ with the topology of smooth Cheeger--Gromov convergence of Ricci flows, uniform in compact subsets of $M\times (-\infty,0)$. 

Let $T_i\nearrow 0$. Since any $(M,g(t),x)_{t\in (-T,0)}$ is the limit of the sequence $(M,g(t+T_i),x)_{t\in (-T-T_i,0]}$, which satisfies \eqref{rfnb},  it follows that $\mathcal{RF}(n,B)=\overline{\mathcal{RF}_{reg}(n,B) }$. 

It is a consequence of estimates of Naber in \cite{Naber10}, as well as the work of Enders \cite{Enders08}, that given a sequence $\{\mathfrak g_i\}_{i}$ and $\mathfrak g$ in   $\mathcal{RF}(n,B) $ such that $\mathfrak g_i \rightarrow \mathfrak g$, the corresponding sequence $l_{\mathfrak g_i}$ converges, up to subsequence, to a limit $l$ in $C^{0,\alpha}_{loc}$. 

Thus we are led to the following definition:

\begin{definition}[Singular reduced distance]\label{def:sin_red_dis}
A function $l:M\times (0,T)\rightarrow \mathbb R$ is a singular reduced distance on $\mathfrak g=(M,g(t),x)_{t\in (-T,0)}\in \mathcal{RF}(n,B)$ if there is a sequence $\mathfrak g_i\in \mathcal{RF}_{reg}(n,B)$ such that $\mathfrak g_i \rightarrow \mathfrak g$ and $l_{\mathfrak g_i}\rightarrow l$ in $C^{0,\alpha}_{loc}$.
\end{definition}
\begin{remark}\label{sin_red_dis_comp}
The estimates in \cite{Naber10} also imply that the collection of the singular reduced distances of a fixed $\mathfrak g\in \mathcal{RF}(n,B)$ is compact in the $C^{0,\alpha}_{loc}$ topology. 
\end{remark}

\subsection{Reduced volume in the singular setting.} Following Definition \ref{def:sin_red_dis} and \eqref{eqn:red_vol} we may define 
\begin{equation}\label{eqn:sin_red_vol_1}
\mathcal V_{\mathfrak g,l}(\tau)=\int_M \frac{e^{-l(y,\tau)}}{(4\pi \tau)^{n/2}} d\mu_{g(-\tau)}(y),
\end{equation}
where $\mathfrak g=(M,g(t),x)_{t\in (-T,0)}\in \mathcal{RF}(n,B)$ and $l$ is a singular reduced distance on $\mathfrak g$.

The curvature bound \eqref{rfnb} and the quadratic growth of a singular reduced distance $l$, again due to \cite{Naber10}, imply that the map $l \mapsto \mathcal V_{\mathfrak g,l}(\tau)$ is continuous, for every $\tau$. Hence, by Remark \ref{sin_red_dis_comp} we may define the singular reduced volume of $\mathfrak g\in \mathcal{RF}(n,B)$ at scale $\tau$ as
\begin{equation}
\mathcal V_{\mathfrak g}(\tau) = \min\{  \mathcal V_{\mathfrak g,l}(\tau), \textrm{$l$ singular reduced distance on $\mathfrak g$}\}.
\end{equation}
\begin{remark}\label{rmk:rv_B}
Note that $\mathcal{RF}(n,B)\subset \mathcal{RF}(n,B')$ for every $B'\geq B$. Thus, the reduced volume $\mathcal V_{\mathfrak g}(\tau)$ may depend on the choice of the constant $B<+\infty$; a larger constant leads to a larger number of competitors in the minimization procedure used to define $\mathcal V_{\mathfrak g}(\tau)$. Nevertheless, we see below that this definition has all the necessary properties we need in our analysis.
\end{remark}

 Before we describe some properties of the reduced volume, recall that a gradient shrinking Ricci soliton is a triplet $(N,g,f)$ where $(N,g)$ is a complete Riemannian manifold and $f\in C^\infty(N)$ satisfies
$$\ric(g)+\hess_g f = \frac{g}{2}.$$
It is a standard fact about gradient shrinking Ricci solitons that there is a constant $c$ such that
$$R+|\nabla f|^2-f=c.$$
We call $(N,g,f)$ a \textit{normalized Ricci soliton} and $f$ a \textit{normalized soliton function} if $c=0$.

Moreover, we will say that $(N,h(t))_{t\in (-\infty,0)}$ is induced by a gradient shrinking Ricci soliton if there exists a normalized soliton function $f\in C^\infty(N)$ such that $(N,h(-1),f)$ is a gradient shrinking Ricci soliton, and the vector field $\nabla f$ is complete.

\begin{lemma}[Proposition 3.1 in \cite{Gianniotis2017}]\label{lem:sin_red_vol_prop}
Given any  $\mathfrak g\in\rf$ the reduced volume $\mathcal V_{\mathfrak g}(\tau)$ has the following properties:
\begin{enumerate}
\item $\mathcal V_{\mathfrak g}(\tau)$ is monotonically decreasing in $\tau$.
\item If $\mathcal V_{\mathfrak g}(\tau_1)=\mathcal V_{\mathfrak g}(\tau_2)$ for some $0<\tau_1<\tau_2$, then for every $\tau$
$$\mathcal V_{\mathfrak g}(\tau)=\mathcal V_{\mathfrak g,l}(\tau)$$
for some singular reduced distance $l$ of $\mathfrak g$. Moreover, $\mathfrak g$ is induced from a shrinking Ricci soliton and $l(\cdot,-1)$ is a normalized soliton function.
\item If there is a sequence $\mathfrak g_i\in\rf$ such that $\mathfrak g_i\rightarrow\mathfrak g$ then
$$\liminf_i \mathcal V_{\mathfrak g_i}(\tau) \geq \mathcal V(\tau),$$
for every $\tau$.
\end{enumerate}
\end{lemma}

\subsection{The density function.}  Using the monotonicity assertion from Lemma \ref{lem:sin_red_vol_prop} we can define the density of $\mathfrak g\in \mathcal{RF}(n,B)$ as
\begin{equation}
\Theta_{\mathfrak g}:= \lim_{\tau\rightarrow 0}\mathcal V_{\mathfrak g}(\tau).
\end{equation}

Moreover, again from Lemma \ref{lem:sin_red_vol_prop}, it follows that if $\mathfrak g_i\rightarrow \mathfrak g$, where $\mathfrak g_i, \mathfrak g\in \rf$, then 
\begin{equation}\label{eqn:density_semicont}
\liminf_i \Theta_{\mathfrak g_i}\geq \Theta_{\mathfrak g}.
\end{equation}

Given a Ricci flow $(M,g(t))_{t\in (-T,0)}$ satisfying \eqref{rfnb}, we now define the \textit{density of $g(t)$ at $x\in M$} as
$$\Theta_g(x)= \Theta_{\mathfrak g_x}.$$

\subsection{Reduced volume and density of shrinking Ricci solitons.} Although the definition of the reduced volume involves minimization over all approximating Ricci flows, which makes it hard to compute, we see below that we can still say enough in the case of shrinking Ricci solitons. This is essentially due to the lower semicontinuity and scaling properties of the reduced volume.

\begin{lemma}[Lemma 3.1 in \cite{Gianniotis2017}]\label{lem:soliton_density}
Let $\mathfrak g=(M,g(t),x)_{t\in(-\infty,0)}\in\rf$ induced by a normalized shrinking Ricci soliton $(M,g(-1),f)$.
\begin{enumerate}
\item $\lim_{\tau \rightarrow \infty} \mathcal V_{\mathfrak g}(\tau)  = \lim_{\tau \rightarrow \infty} \mathcal V_{\mathfrak g,l}(\tau) = \int_M (4\pi)^{-\frac{n}{2}} e^{-f}d\mu_{g(-1)}$, for any singular reduced distance $l$ of $\mathfrak g$.
\item If $x$ is a critical point of $f$, then 
$$\Theta_g (x)=\int_M (4\pi)^{-\frac{n}{2}} e^{-f}d\mu_{g(-1)} \leq \Theta_g (y),$$ for any $y\in M$.
\item If a singular reduced distance $l$ of $\mathfrak g$ is a soliton function then
$$\mathcal V_{\mathfrak g}(\tau)=\mathcal V_{g,l}(\tau),$$
for every $\tau$.
\end{enumerate}
\end{lemma}

\subsection{Tangent flows and density.} Let $\mathfrak h \in \mathcal{RF}(n,B)$ be a tangent flow of $\mathfrak g\in \mathcal{RF}(n,B)$, namely the limit of a sequence $(\mathfrak g)_{s_i}$, for $s_i\searrow 0$. By \cite{Naber10}, $\mathfrak h$ is induced by a gradient shrinking Ricci soliton. The following theorem is proven in \cite{Gianniotis2017}:

\begin{theorem}[Theorem 5.1 in \cite{Gianniotis2017}]\label{thm:density_uniqueness}
Let $(N,h(-1),f)$ be the shrinking Ricci soliton associated to $\mathfrak h$, with $f$ being a normalized soliton function. Then 
\begin{equation}
\Theta_{\mathfrak g}=\Theta_{\mathfrak h}=\int_N  (4\pi)^{-\frac{n}{2}} e^{-f}d\mu_{h(-1)}.
\end{equation}
\end{theorem}
It follows that, although not unique, any tangent flow of $\mathfrak g$ has the same asymptotic reduced volume $\lim_{\tau\rightarrow+\infty} \mathcal{V}_{\mathfrak h}(\tau)$, by Lemma \ref{lem:soliton_density}.

We describe below another important implication of  Theorem \ref{thm:density_uniqueness}: although the reduced volume may depend on $B$, as was discussed in Remark \ref{rmk:rv_B}, the density is independent of such choice. This follows from the observation that the collection of tangent flows $\mathfrak h$ does not depend on $B$. Hence, the corresponding asymptotic reduced volume is independent of $B$, thus from Theorem \ref{thm:density_uniqueness} the density $\Theta_{\mathfrak g}$ also does not depend on $B$.

\subsection{The spine of a shrinking Ricci soliton.} Let $(N,h(-1),f)$ be a gradient shrinking Ricci soliton with bounded curvature, and associated Ricci flow $(N,h(t))_{t\in (-\infty,0)}$. It is easy to see that this flow satisfies \eqref{rfnb}, for some $B<+\infty$.

The discussion above shows that the density function $\Theta_h :N\rightarrow (0,1]$ is well defined and independent of the choice of the class $\mathcal{RF}(n,B)$.

We can thus define the \textit{spine of  $(N,h(t))_{t\in(-\infty,0)}$} as
$$S(N,h)=\{ x\in N, \Theta_h \textrm{ attains its minimum value at } x\}.$$
We note that $S(N,h)$ is non-empty, since $\Theta_h$ attains a minimum value at any critical point of $f$, by Lemma \ref{lem:soliton_density}. Due to the quadratic growth of $f$, see for instance \cite{HaslhoferMuller11}, $f$ always has a critical point. 

Moreover, the lower semicontinuity of the density function \eqref{eqn:density_semicont} implies that $S(N,h)$ is a closed subset of $N$.

The notion of the spine $S(N,h)$ will be important to us because of the following theorem.

\begin{theorem}[Theorem 4.1 in \cite{Gianniotis2017}]\label{thm:spine_geometry}
Let $(N,h(t))_{t\in (-\infty,0)}$ be the Ricci flow induced by a non-flat gradient shrinking Ricci soliton satisfying \eqref{rfnb}. Then, there exists an integer $2\leq k \leq n$, a constant $D(n,B)<+\infty$, and a gradient shrinking Ricci soliton $(\bar N,\bar h(t))_{t\in (-\infty,0)}$ such that
\begin{enumerate}
\item $(N,h(t))$ splits isometrically as $(\bar N,\bar h(t)) \times \mathbb R^{n-k},g_{Eucl})$.
\item $S(N,h)=K\times \mathbb R^{n-k}$ and $\diam_{\bar h(t)}(K)\leq D\sqrt{-t}$ for every $t\in (-\infty,0)$.
\end{enumerate}
\end{theorem}
\begin{remark}
Observe that if $k =1$ above $(N,h(t))$ is necessarily the Euclidean space.
\end{remark}
\begin{remark}
Note that in the regularity theory for harmonic maps/minimal currents, the spine is defined as the set where the density attains its maximum,  in contrast to the definition above where the spine consists of points with minimal density. This is due to the reversal of the monotonicity and semicontinuity properties of the reduced volume.

Recall that the spine of a tangent map/cone is also the linear subspace of the available translation symmetries. Theorem \ref{thm:spine_geometry} can be viewed as the analogue of this fact for shrinking Ricci solitons with bounded curvature, as it implies that the spine $(S(N,h),h(t))$ converges to $\mathbb R^{n-k}$ in the pointed Gromov--Hausdorff topology, as $t\rightarrow 0$. 
\end{remark}

\begin{remark}
If $(N,h(t))_{t\in(-\infty,0)}$ is the flow induced by a compact shrinking Ricci soliton, the tangent flow at any $x\in N$ is $(N,h(t),x)_{t\in(-\infty,0)}$. This implies that the density function $\Theta_h$ is constant, hence $S(N,h)=N$.

The same holds if $N=\bar N \times \mathbb R^k$ for some compact shrinking Ricci soliton $(\bar N,g,f)$, as for example $\mathbb S^{n-k}\times\mathbb R^k$ with the standard soliton structure.

For the $U(n)$-invariant shrinking K\"ahler Ricci solitons  on line bundles over $\mathbb C \mathbb P^{n-1}$ constructed in \cite{FIK03}, the spine is the zero section $\mathcal Z$ of the corresponding line bundle. This is because the flow is non-singular away from $\mathcal Z$ and $U(n)$ acts isometrically and transitively on $\mathcal Z$.
\end{remark}

\subsection{Compactness of shrinking solitons.} Below we prove a compactness theorem for Ricci solitons, under a uniform curvature bound. Moreover, Lemma \ref{lem:soliton_compactness} asserts that, along a convergent sequence of such solitons, points with lowest density converge to points in the spine of the limit. 

We first need the following auxiliary lemma, which allows to center soliton functions `around'  a given point on the spine. 

\begin{lemma}[Aligning a soliton function to a point on the spine]\label{lem:align_spineandcritical_points}
Let $(N,g,f)$ be a gradient shrinking Ricci soliton satisfying
\begin{equation*}
\sup_{N} |\riem(g)|_g \leq B,
\end{equation*}
for some $B<+\infty$ and $f$ is a normalized soliton function. Also, let $q\in S(N,g)$. Then, there is a normalized soliton function $f'$ with a critical point $p\in N$ such that
\begin{equation}\label{eqn:adjusted_point}
d_g(p,q) \leq D,
\end{equation}
where $D=D(n,B)<+\infty$ is the constant given by Theorem \ref{thm:spine_geometry}.
\end{lemma}
\begin{proof}
 By Theorem \ref{thm:spine_geometry}, $(N,g)$ splits isometrically as $(\bar N^k,\bar g)\times\mathbb R^{n-k}$ and $S(N,g)= K \times \mathbb R^{n-k}$, where $K$ is compact and satisfies
 \begin{equation}\label{diameter}
 \diam_{\bar g} (K) \leq D.
 \end{equation} 
 
We may assume that $q=(\bar q,0)$ for some $\bar q \in K$. 

Now, let $p_0=(\bar p, v_0)\in \bar N\times \mathbb R^k$ be a critical point for $f$ and define $f'$ by 
$$f'(\bar x, v) = f( \bar x, v+v_0).$$
 Note that $f'$ is also a normalized soliton function and has a critical point at $(\bar p,0)$.

Lemma \ref{lem:soliton_density} implies that critical points of $f'$ are in $S(N,g)$, thus $\bar p\in K$. Then, \eqref{diameter} implies that 
$$d_{g}(p,q) = d_{\bar g}(\bar p,\bar q)\leq D.$$ 
\end{proof}

\begin{lemma}\label{lem:soliton_compactness}
Let $(N_i,h_i(t), q_i)_{t\in (-\infty,0)}$ be a sequence of pointed complete Ricci flows induced by gradient shrinking Ricci solitons with bounded curvature.  Suppose that
\begin{equation}\label{sol_boun}
\sup_{ N_i }|\riem(h_i(t))|_{h_i(t)}\leq \frac{B}{|t|}
\end{equation}
for $t\in (-\infty,0)$ and
\begin{equation}\label{inj_bound}
inj_{h_i(-1)}(q_i)\geq i_0.
\end{equation}
Then, there exists a subsequence $(N_{i_l},h_{i_l}(t),q_{i_l})_{t\in (-\infty,0)}$ converging in the smooth Cheeger--Gromov topology to $(N_\infty,h_\infty(t),q_\infty)_{t\in (-\infty,0)}$, which also satisfies \eqref{sol_boun} and is induced by a shrinking Ricci soliton with bounded curvature. 

 Moreover, if $q_i\in S(N_i,h_i)$, then $q_\infty\in S(N_\infty,h_\infty)$ and $$\Theta_{h_\infty}(q_\infty)=\lim_l \Theta_{h_i}(q_i).$$
\end{lemma}
\begin{proof}
In view of bounds \eqref{sol_boun}-\eqref{inj_bound} and Hamilton's compactness theorem for sequences of Ricci flows, passing to a subsequence if necessary,  we may assume that $(N_i,h_i(t), q_i)_{t\in (-\infty,0)}$ converges to a limit flow $(N_\infty, h_\infty, q_\infty)_{t\in (-\infty,0)}$ that also satisfies \eqref{sol_boun}.

Now, suppose that $q_i\in S(N_i,h_i)$ and let $f_i$ be normalized soliton functions with critical points $p_i\in N_i$, satisfying 
\begin{equation}\label{db}
d_{h_i}(p_i,q_i)\leq D,
\end{equation}
 given by Lemma \ref{lem:align_spineandcritical_points}. 

Since $p_i$ is a critical point, \eqref{sol_boun} implies
\begin{equation}
|f_i(p_i)|=\big|R(h_i(-1))(p_i) + |\nabla f_i(p_i)|^2 \big| \leq C(n,B).
\end{equation}
Differentiating the soliton equation and applying Shi's derivative estimates we obtain uniform bounds on $f_i$ and its derivatives, within bounded distance from $p_i$. Thus, by Arzela--Ascoli and passing to a subsequence if necessary, using \eqref{db}, we may assume that $f_i$ converges smoothly to a  function $f_\infty$ on $N_\infty$, uniformly locally. Moreover, $f_\infty$ is a normalized soliton function, since it is a property that passes to smooth limits.

Since $f_i$ grow quadratically in the distance from $p_i$, and the volume of $\vol_{h_i(-1)} (B_{p_i}(r))$ grows at most exponentially in $r$, it follows that
\begin{equation}\label{eqn:integrals_converge}
\int_{N_{i}}e^{-f_{i}}d\mu_{h_{i}(-1)}\rightarrow \int_{N_\infty} e^{-f_\infty}d\mu_{h_\infty(-1)}.
\end{equation}

Since $q_i\in S(N_i,h_i)$, recall that $\Theta_{h_i}(q_i)=\int_{N_{i}}e^{-f_{i}}d\mu_{h_{i}(-1)}$, due to Lemma \ref{lem:soliton_density}. The lower semicontinuity of the density function under the Cheeger--Gromov convergence of Ricci flows and \eqref{eqn:integrals_converge} imply that
\begin{equation}
 \int_{N_\infty} e^{-f_\infty}d\mu_{h_\infty(-1)}\geq\Theta_{h_\infty}(q_\infty) .
\end{equation}
On the other hand, monotonicity of reduced volume and Lemma \ref{lem:soliton_density} implies that
\begin{equation}
 \int_{N_\infty} e^{-f_\infty}d\mu_{h_\infty(-1)}= \lim_{\tau\rightarrow +\infty} \mathcal V_{\mathfrak q_{\infty}}(\tau)\leq \Theta_{h_\infty}(q_\infty).
\end{equation}
Thus, $\Theta_{h_\infty}$ attains its minimum value $\int_{N_\infty} e^{-f_\infty}d\mu_{h_\infty(-1)}=\lim_i \Theta_{h_i}(q_i)$ at $q_\infty$, hence $q_\infty\in S(N_\infty,h_\infty)$. This suffices to prove the Lemma.
\end{proof}

\section{The quantitative stratification}\label{section:q_stratification}
In this section we adapt the ideas of Cheeger--Naber from \cite{CheegerNaber13a} to the setting of Ricci flows subject to a Type I curvature bound. In particular, we define the quantitative stratification and prove volume estimates similar to those in \cite{CheegerNaber13a}.

Before we define the quantitative stratification in detail, a few definitions are in order. First, we need an appropriate notion of `closeness' of two pointed Ricci flows:

\begin{definition}\label{distance}
Let $\mathfrak g_1= (M_1,g_1(t),p_1 )_{t\in (-2,0)}, \mathfrak g_2= (M_2,g_2(t), p_2)_{t\in (-2,0)}$ be complete pointed Ricci flows. We say that $\mathfrak g_2$ is $\eta$-close to $\mathfrak g_1$, $\eta>0$, if the following holds:
\begin{enumerate}
\item There exists $U\subset M_1$ with $B_{g_1(-1)} (p_1, \eta^{-1})\subset U$ and a smooth map $F: U\rightarrow M_2$, diffeomorphism onto its image, satisfying $F(p_1)=p_2$.
\item  $ (1+\eta)^{-2} g_1(t) \leq F^* g_2(t) \leq (1+\eta)^2 g_1(t)$ for every $t\in [-2+\eta,-\eta]$.
\item $ |(\nabla^{g_1(t)})^l F^* g_2(t)|_{g_1(t)} <\eta $ for $t\in [-2+\eta,-\eta]$ and $1\leq l \leq \lfloor 1/\eta\rfloor$.
\end{enumerate}
\end{definition}

Recall that from the work of Naber \cite{Naber10}, tangent flows are  selfsimilar solutions to the Ricci flow, induced by shrinking Ricci solitons. In other words, the flow looks selfsimilar in small scales. The definition below makes this precise, and also quantifies the amount of translational symmetry of a given Ricci flow, in the sense of isometric splitting of Euclidean factors.

\begin{definition}
Given $\varepsilon>0$, $r\in (0,1]$, $B<+\infty$ and integer $k\geq0$, a Ricci flow $\mathfrak g_x=(M,g(t),x)_{t\in (-2,0)}$ is $(\varepsilon, r, k,B)$-selfsimilar with respect to the $k$-dimensional subspace $V\subset T_x M$  if there exists a pointed shrinking Ricci soliton
$$\mathfrak h=(N,h(t),q)_{t\in (-\infty,0)} = (\tilde N,\tilde h(t),\tilde q)\times (\mathbb R^k, g_{Eucl},0)$$
 satisfying  $\sup_N|\riem(h(-1))|_{h(-1)} \leq B$, such that $q\in S( N, h)$, $( \mathfrak g_x )_r$ is $\varepsilon$-close to $\mathfrak h$ and $V=F_*( \{0\} \times \mathbb R^k)$, where $F_* :T_{\tilde q} \tilde N\times \mathbb R^k \rightarrow T_p M$, $F$  as in Definition \ref{distance}.
\end{definition}

Including a uniform \textit{global} curvature bound for the soliton in the definition above is an unusual feature, compared to other instances of quantitative stratification. Here, it provides the essential control on the geometry of the spine, by Theorem \ref{thm:spine_geometry}. 

From now on we fix a $B<+\infty$ and define the quantitative stratification as follows:

\begin{definition}
Let $(M,g(t))_{t\in (-2,0)}$ be a complete Ricci flow, $\dim M=n$. Given an integer $0\leq k \leq n$, $\eta>0$ and $\tau \in (0,1]$ define $S^k_{\eta,\tau}\subset M$ as follows:
\begin{align*}
S^k_{\eta,\tau}=&\{x\in M,\;  \mathfrak g_x \;\tm{is not  $(\eta,s,k+1, B)$-selfsimilar for any $s\in [\tau^{1/2},1]$}\}.
\end{align*}
\end{definition}

Note that the following inclusions hold when $k' \geq k$, $\eta'\leq \eta$ and $\tau'\geq \tau$:
\begin{equation*}
S^k_{\eta,\tau} \subset S^{k'}_{\eta',\tau'}.
\end{equation*}

Moreover, applying Lemma \ref{lem:soliton_compactness} we easily see that the quantitative stratification $S^k_{\eta,\tau}$ is related to the stratification $\Sigma_k$ of the singular set $\Sigma$ of $(M,g(t))_{t\in(-2,0)}$  by
\begin{equation*}
\Sigma_k =  \bigcup_{\eta} \bigcap_{\tau}   S^k_{\eta,\tau}.
\end{equation*}
The aim of this section is to prove Theorem \ref{thm:vol_estimate}.

\subsection{Almost self-similar scales.} In this section we see that the scales and points around which a Ricci flow in $\mathcal{RF}(n,B)$ looks selfsimilar are characterized by the associated reduced volume being `almost' constant. We then show,  in Lemma \ref{lem:line_up}, that as the flow evolves such points are locally `attracted' towards a lower dimensional submanifold.

\begin{lemma}[Quantitative rigidity]\label{lem:quant_rigidity}
For every $\varepsilon, \kappa>0$ and $B<+\infty$, there exists $0<\delta_1(\varepsilon, \kappa,B)\leq\varepsilon$ such that if $\mathfrak g=(M,g(t),x)_{t\in (-2,0)}\in\rf$ satisfies
\begin{enumerate}
\item $g(t)$ is $\kappa$ non-collapsed below scale $1$ for all $t\in (-2,0)$,
\item $\mathcal V_{\mathfrak g} (\delta_1 r^2)- \mathcal V_{\mathfrak g}(r^2) <\delta_1$, for some $r\in (0,1]$
\end{enumerate}
then $\mathfrak g$ is $(\varepsilon, r, 0,B)$-selfsimilar.
\end{lemma}

\begin{proof}
Fix $\varepsilon, \kappa>0$, $\gamma\in (0,1)$ and $B<+\infty$. Suppose there is a sequence $\mathfrak g_i=(M_i,g_i(t),x_i)_{t\in (-2,0)}\in \mathcal{RF}(n,B)$ that is $\kappa$ non-collapsed below scale $1$, and sequences $\delta_i\searrow 0$, $\delta_i<1/2$, and  $r_i\in (0,1]$  such that 
\begin{equation}
\mathcal V_{\mathfrak g_i} (\delta_i r_i^2)- \mathcal V_{\mathfrak g_i}(r_i^2) <\delta_i, \label{small_energy}
\end{equation}
but $\mathfrak g_i$ is not $(\varepsilon, r_i,0,B)$-selfsimilar.

The curvature bound of the class $\mathcal{RF}(n,B)$ and the $\kappa$ non-collapsing \linebreak assumption imply that a subsequence of $(\mathfrak g_i)_{r_i}$ converges to a complete pointed Ricci flow $\mathfrak h=(N,h(t),q)_{t\in (-2,0)}\in \mathcal{RF}(n,B)$ in the smooth Cheeger--Gromov topology.

Let $l_i$ be a singular reduced distance function of $(\mathfrak g_i)_{r_i}$ that realizes the reduced volume at scale $1/2$, namely 
$$\mathcal V_{(\mathfrak g_i)_{r_i}}( 1/2 )=\mathcal V_{(\mathfrak g_i)_{r_i} , l_i }(1/2).$$
From the estimates of Naber \cite{Naber10}, a subsequence of $l_i$ converges to a singular reduced distance $l_\infty$ of $\mathfrak h$, thus a subsequence of $\mathcal V_{ (\mathfrak g_i)_{r_i} }(1/2)$ converges to $\mathcal V_{\mathfrak h, l_\infty}(1/2)$. Moreover, for the same reason a subsequence of $\mathcal V_{ (\mathfrak g_i)_{r_i}, l_i }(1)$ converges to $\mathcal V_{\mathfrak h, l_\infty}(1)$.

Hence, from monotonicity and \eqref{small_energy}, it follows that
\begin{equation}
\mathcal V_{\mathfrak h,l_{\infty}}(1)\leq \mathcal V_{\mathfrak h,l_{\infty}}(1/2) \leq \mathcal V_{\mathfrak h,l_{\infty}}(1),
\end{equation}
since by the definition of the singular reduced volume 
$$\mathcal V_{ (\mathfrak g_i)_{r_i} }(1)\leq \mathcal V_{(\mathfrak g_i)_{r_i}, l_i}(1). $$
Thus, $l_\infty$ is a normalized soliton function and $\mathfrak h$ is a shrinking Ricci soliton,  by Lemma \ref{lem:sin_red_vol_prop}. Moreover, the underlying Ricci flow of $\mathfrak h$ satisfies the Type I bound \eqref{rfnb}. This contradicts the assumption that $\mathfrak g_i$ is not $(\varepsilon, r_i,0,B)$-selfsimilar.
\end{proof}
\begin{remark}
Note that in the proof of Lemma \ref{lem:quant_rigidity} we do not use the full strength of assumption (2) and in fact the lemma holds under the weaker hypothesis 
\begin{equation*}
\mathcal V_{\mathfrak g} (\gamma r^2)- \mathcal V_{\mathfrak g}(r^2) <\delta_1, 
\end{equation*}
for some $r\in (0,1]$ and $\gamma\in(0,1)$, and small enough $\delta_1$, with the same proof. However, the current proof uses the forward \cite{ChenZhu06} and backward \cite{Kotschwar10} uniqueness property of complete Ricci flows with bounded curvature in an essential way, namely to assert that in part (2) of Lemma \ref{lem:sin_red_vol_prop} the flow $\mathfrak g$ is a shrinking soliton for all time.

The weaker statement of Lemma \ref{lem:quant_rigidity} is more likely to hold in the incomplete setting, and it suffices for the arguments of this section.
\end{remark}

\begin{lemma}[Almost splitting]\label{lem:almost_splitting}
For every $\varepsilon, \lambda, \mu, \kappa>0$, $\gamma\in (0,1]$  and $ B<+\infty$, there exists $0<\delta_2(\varepsilon, \lambda,\mu,\kappa, B,\gamma)\leq \varepsilon$ such that, if $(M,g(t),x_1)_{t\in (-2,0)}\in\rf$, $g(t)$ is $\kappa$ non-collapsed below scale $1$ for every $t\in(-2,0)$ and for some $r\in(0,1]$
\begin{enumerate}
\item $(M,g(t), x_1)$ is $(\delta_2, r, k,B)$-selfsimilar at $x_1$ with respect to $V\subset T_{x_1} M$, for some $0\leq k\leq n$,
\item $(M,g(t),x_2)$ is $(\delta_2,r,0,B)$-selfsimilar,
\item $d_{g(-r^2)}(x_1,x_2)<\lambda r$, 
\item $d_{g(- \tau)}(x_2, \exp_{g(-\gamma r^2),x_1}(V\cap B_0(2\lambda r))\geq (D+\mu)\sqrt{\tau}$\\ for some $\tau\in [r^2 \mu,r^2 (2-\mu)]$, where $D=D(n,B)$ is the constant given by Theorem \ref{thm:spine_geometry},
\end{enumerate}
then $(M,g(t),x_1)_{t\in (-2,0)}$ is $(\varepsilon, r, k+1,B)$-selfsimilar. 
\end{lemma}
\begin{proof}
Fix $\varepsilon, \lambda, \mu, \kappa>0$,  $\gamma\in(0,1]$ and $B<+\infty$,  as in the statement of the theorem. Suppose there are sequences $\delta_i\searrow 0$ and $r_i\in(0,1]$, and a sequence of Ricci flows $(M_i,g_i(t))_{t\in (-2,0)}\in \mathcal{RF}(n,B)$, $\kappa$ non-collapsed below scale $1$, satisfying:
\begin{enumerate}
\item $(M_i, g_i (t), x_1^i)$ is $(\delta_i, r_i, k,B)$-selfsimilar with respect to $V_i\subset T_{x_1^i} M_i$,
\item $(M_i, g_i (t), x_2^i)$ is $(\delta_i, r_i, 0,B)$-selfsimilar,
\item $d_{g_i(-r_i^2)} (x_1^i, x_2^i) < \lambda r_i$,
\item $d_{g_i(-\tau_i)}(x_2^i, \exp_{g_i(-\gamma r_i^2), x_1^i}(V_i\cap B_0(2\lambda r_i)) ) \geq (D+\mu)\sqrt{\tau_i}$ for some \linebreak $\tau_i\in [r_i^2 \mu,r_i^2 (2-\mu)]$,
\end{enumerate}
but such that $(M_i, g_i(t),x_1^i)$ is not $(\varepsilon, r_i, k+1,B)$-selfsimilar.

Since $\delta_i\searrow 0$,  assumption (1) above and Lemma \ref{lem:soliton_compactness} imply that we may assume, by passing to subsequence if necessary, that $(M_i, r_i^{-2}g_i(r_i^2 t), x_1^i)$ converges in the smooth pointed Cheeger--Gromov topology to a shrinking Ricci soliton
 $$(N,h(t),q_1)_{t\in(-\infty,0)} = (\tilde N,\tilde h(t),\tilde q_1) \times (\mathbb R^k, g_{Eucl},0).$$
with $q_1\in S(N, h)$, which satisfies $\sup_N |\riem(h(-1))|_{h(-1)} \leq B$.

Moreover, since $(M_i, g_i(t),x_1^i)$ is not $(\varepsilon, r_1, k+1,B)$-selfsimilar, it follows that $(\tilde N,\tilde h(t))$ does not split any Euclidean factors. Then, by Theorem \ref{thm:spine_geometry},  $S(N,h)=K\times \mathbb R^k$, where $K\subset \tilde N$ is compact and satisfies
\begin{equation}\label{diam_K}
\diam_{\tilde h(-\tau)} (K)\leq D \sqrt{\tau},
\end{equation}
 for every $\tau\in (0,+\infty)$.
 
Similarly, we may assume that $(M_i, r_i^{-2}g_i(r_i^2 t), x_2^i)$ converges to a shrinking Ricci soliton $(\hat N,\hat h(t),\hat q)$, with $\hat q\in S(\hat N, \hat h)$.
 
 Since $d_{g_i(-r_i^2)}(x_1^i,x_2^i) <\lambda r_i$, the flows $(N, h(t))$ and $(\hat N,\hat h( t))$ are isometric, by the uniqueness of smooth limits, so from now on we will identify them. In particular, we identify $\hat q$ with $q_2\in N$. 
 
  Then, since $q_1,q_2\in S(N,h)=K\times \mathbb R^k$, let $q_1=(\tilde q_1,0)$ and  $q_2=(\tilde q_2,v)$, $v\in \mathbb R^k$.
 
 Now, if $\Phi_i: B_{h(-1)}(q_1, R_i) \rightarrow M_i$, where $R_i\rightarrow+\infty$, are diffeomorphisms associated to the convergence, then
 $$\Phi_i^{-1} ( \exp_{g_i(-\gamma r_i^2),x_1^i}(V_i )  ) \rightarrow \{\tilde q_1 \} \times \mathbb R^k,$$
 smoothly and uniformly on compact sets. Moreover, $\Phi_i^{-1}( x_2^i) \rightarrow q_2\in N$ and $\tau_i\rightarrow \bar\tau$, up to subsequence.
 
Since $q_2=( \tilde q_2,v)\in K\times \mathbb R^k$,  by \eqref{diam_K} we conclude that 
\begin{equation}\label{eqn:contradicted}
d_{h(-\bar\tau)}(q_2, \{\tilde q_1\}\times \mathbb R^k) = d_{\tilde h(-\bar \tau)} (\tilde q_1,\tilde q_2)\leq D\sqrt{\bar \tau},
\end{equation}
since the splitting $N=\tilde N\times\mathbb R^k$ is isometric.

On the other hand, $d_{g_i(-\tau_i)}(x_2^i, \exp_{g_i(-\gamma r_i^2), x_1^i}(V_i\cap B_0(2\lambda r_i)) ) \geq (D+\mu)\sqrt{\tau_i}$ implies that 
$$d_{h(-\bar\tau)}(q_2, \{\tilde q_1\}\times \mathbb R^k  ) \geq (D+\mu)\sqrt{\bar\tau},$$
 which contradicts \eqref{eqn:contradicted}.
\end{proof}

\begin{lemma}[Line-up lemma]\label{lem:line_up}
Let $\mathfrak g_x:=(M,g(t),x)_{t\in(-2,0)}\in\rf$ such that $g(t)$ is $\kappa$ non-collapsed below scale $1$ for every $t\in(-2,0)$. Then, for every $\lambda,\mu,\nu>0$ and $\gamma\in(0,1)$ there exists $\delta_3(B,\gamma,\kappa,\lambda,\mu,\nu)>0$ such that if 
\begin{equation}\label{almost_0_similar}
\mathcal V_{\mathfrak g_x} (\delta_3 \bar \tau) - \mathcal V_{\mathfrak g_x} (\bar \tau)<\delta_3,
\end{equation}
for some $\bar\tau\in(0,1]$, then there exists $0\leq k \leq n$ and a $k$-dimensional subspace $V$ of $T_x M$ such that
\begin{enumerate}
\item $\mathfrak g_x$ is $(\nu, \bar\tau^{1/2}, k,B)$-selfsimilar with respect to $V$.
\item The set 
$$ L_{\bar\tau,\delta_3}= \{  y\in M,\;\mathcal V_{\mathfrak g_y} (\delta_3 \bar\tau) - \mathcal V_{\mathfrak g_y} ( \bar\tau) <\delta_3  \}$$
satisfies 
\begin{equation}\label{lineup}
 L_{\bar\tau,\delta_3} \cap B_{g(- \bar\tau)}(x, \lambda \bar\tau^{1/2})  \subset T^{g(-\tau)}_{(D+\mu) \sqrt \tau} \big( \exp_{g(-\gamma \bar\tau),x} (V\cap B_0(2\lambda\bar\tau^{1/2})) \big),
\end{equation}
for every $\tau\in [\mu \bar\tau, (2-\mu) \bar\tau]$, where $D$ is the constant given by Theorem \ref{thm:spine_geometry}. \\ Here $T^g_r(S)$ denotes the $r$-tubular neighbourhood of a set $S$ with respect to the Riemannian metric $g$.
\end{enumerate}
\end{lemma}

\begin{proof}
 Let $\delta(\nu)=\delta_2 (\nu, \lambda,\mu, \kappa, B, \gamma)\leq\nu$, where $\delta_2$ is given by Lemma \ref{lem:almost_splitting}, and set $a_i(\nu) = \delta\circ \ldots \circ \delta (\nu)\leq \nu$, where the composition is taken \linebreak$i$-times. Then, choose $\delta_3=\delta_1(a_n(\nu), \kappa,B)$, where $\delta_1$ is given by Lemma \ref{lem:quant_rigidity}. Thus, by \eqref{almost_0_similar}, it follows that $\mathfrak g_x$ is $(a_n,\bar \tau^{1/2},0,B)$-selfsimilar.

 Let $k$ be the maximum integer such that $0\leq k \leq n$ and  $\mathfrak g_x$ is $(a_{n-k},\bar\tau^{1/2},k,B)$-selfsimilar with respect to some $V^k\subset T_x M$. 
 
 Suppose that \eqref{lineup} doesn't hold for some $\tau\in [\mu \bar\tau, (2-\mu)\bar\tau]$. Thus, there is $y\in B_{g(- \bar\tau)}(x,\lambda \bar\tau^{1/2})$ with $\mathcal V_{\mathfrak g_y} (\delta_3\bar\tau) - \mathcal V_{\mathfrak g_y} ( \bar\tau)<\delta_3$ but 
$$d_{g(-\tau)}(y, \exp_{g(-\gamma \bar\tau),x }(V \cap B_0(2\lambda \bar\tau^{1/2})))\geq (D+\mu)\sqrt\tau.$$
By Lemma \ref{lem:quant_rigidity} it is also true that $\mathfrak g_y$ is $(a_n,\bar \tau^{1/2},0,B)$-selfsimilar. It then follows by Lemma \ref{lem:almost_splitting} that $\mathfrak g_x$ is $(a_{n-(k+1)},\bar\tau^{1/2},k+1,B)$-selfsimilar, which is a contradiction.
\end{proof}
\begin{remark}
Although the arguments in Lemmata \ref{lem:almost_splitting} and \ref{lem:line_up} are very similar to other instances of quantitative stratification \cite{CheegerNaber13a, CheegerNaber13b, ChHasNab13,ChHasNab15,BreinerLamm14}, it is interesting to point out how Lemma \ref{lem:line_up} differs. 

In \cite{CheegerNaber13a, CheegerNaber13b, ChHasNab13,ChHasNab15,BreinerLamm14} the selfsimilar points line up close to a lower dimensional subspace. Taking the analogy to the Ricci flow naively, one might expect that selfsimilar points will tend to line up around a lower dimensional submanifold. However, this is certainly not true for the Ricci flow, as the example of the standard Ricci flow on the cylinder $\mathbb S^2\times \mathbb R$, becoming singular at $t=0$, shows: there, \textit{every} point is selfsimilar, but the diameter of the $\mathbb S^2$ factor is small only for times near $t=0$. This example illustrates that a statement like that of Lemma \ref{lem:line_up} is more likely to hold.
\end{remark}

\subsection{Energy decomposition.}\label{subsec:energy_decomp} Let $(M,g(t))_{t\in (-2,0)}$ be a complete Ricci flow with bounded curvature satisfying:
\begin{itemize}
\item $|\riem(g(t))|_{g(t)}\leq B/|t|$ on $M\times (-2,0)$,
\item $g(t)$ is $\kappa$ non-collapsed below scale $1$.
\end{itemize}

For every $x\in M$ and $0<\tau_1<\tau_2\leq 1$ define 
\begin{eqnarray*}
\mathcal W_{\tau_1,\tau_2} (x)= \mathcal V_{\mathfrak g_{x}}(\tau_1) - \mathcal V_{\mathfrak g_{x}}(\tau_2) \geq 0.
\end{eqnarray*}

Let $\alpha,\gamma \in (0, 1)$ and $\delta>0$, and set $\tau_i = \gamma^i \alpha$. Then, for every $x\in  M$ define the sequence $$T(x):=(T_1(x), T_2(x),\ldots\;) \in \{0,1\}^\mathbb N$$ as
\begin{equation*}
T_i(x):= \left\{
\begin{array}{ll}
1, & \mathcal W_{\delta \tau_{i-1},\tau_{i-1}} (x) \geq \delta, \\  
0, & \mathcal W_{\delta \tau_{i-1},\tau_{i-1}}  (x) < \delta.
\end{array}
\right.
\end{equation*}

Now, given any $\textbf a=(a_1,\ldots, a_j)\in \{0,1\}^j$, for some integer $j\geq 1$, define $E_{\textbf a}\subset M$ as follows:
$$ E_{\textbf a}=\{x \in  M,\; T_i(x) =a_i\quad \mathrm{for\; every}\quad 1\leq i \leq j\}.$$

\subsection{Quantitative differentiation.}\label{subsec:quant_diff}
A priori there are $2^j$ sets of the form $E_{\textbf a}$, for $\textbf a\in\{0,1\}^j$. We will see below that there is in fact a much smaller number of such sets, which grows polynomially in $j$.

 Let $m\geq 1$ be the minimum integer so that the intervals $[\gamma^{m(i-1)}\alpha, \delta \gamma^{m(i-1)}\alpha)$, for all integers $i\geq 1$, are disjoint. Namely $m= \big\lceil\frac{\log \delta}{\log \gamma}\big\rceil$. Since
\begin{equation*}
\sum_{i=1}^\infty \mathcal W_{\delta\gamma^{m(i-1)}\alpha,\gamma^{m(i-1)}\alpha} (x)\leq \Theta_g(x) -\mathcal V_{\mathfrak g_x}(\alpha) \leq 1,
\end{equation*}
it follows that the number of non-negative integers $i$ for which 
$$\mathcal W_{\delta\gamma^{m(i-1)} \alpha,\gamma^{m(i-1)}\alpha} (x)\geq\delta$$ 
is at most $\lfloor 1/\delta \rfloor$, hence the number of integers $i$ for which
$$\mathcal W_{\delta\gamma^{i-1} \alpha,\gamma^{i-1}\alpha} (x)\geq\delta$$ 
is at most $m\lfloor1/\delta \rfloor$.

Thus, for each $x\in M$,  $T_i(x)=1$ for at most $K(\delta,\gamma)=m\lfloor 1/\delta \rfloor$ values of $i$. This implies that for $j\geq K$ there are only 
\begin{equation}
\left(  
\begin{array}{cc}
j \\
K
\end{array}
\right) \leq j^K
\end{equation}
disjoint sets $E_{\textbf a}$ for $\textbf a\in \{0,1\}^j$. Thus, for any $j\geq 1$ there are at most $2j^K$ disjoint such subsets.
 
\subsection{Covering lemma.} 
Let $(M,g(t))_{t\in (-2,0)}$ be a complete Ricci flow belonging to the class $\mathcal C(n,B,\kappa_0,\kappa_1)$. 

\begin{lemma}\label{lem:vol_control}
Given $\alpha\leq 1$, there exists a $\kappa_2(\alpha,B,\kappa_0)>0$  such that for every $x\in M$, and $r\leq \gamma^{l/2}$, $\tau_l=\gamma^l\alpha$ for any $l\geq 0$:
\begin{equation}\label{eqn:volume_bounds}
\kappa_2 r^n \leq  \vol_{g(-\tau_l)}(B_{g(-\tau_l)}(x,r)) \leq \kappa_1 r^n.
\end{equation}
\end{lemma}
\begin{proof}
We will first prove the lower bound. The curvature bound of the class $\mathcal C(n,B,\kappa_0,\kappa_1)$ implies that for every $l\geq0$
$$|\riem(g(-\tau_l)) |_{g(-\tau_l)} \leq \frac{B}{\gamma^l \alpha},$$
hence the $\kappa_0$ non-collapsing property implies that for every scale $r$ small enough so that $\frac{B}{\gamma^l \alpha}\leq \frac{1}{r^2}$ 
\begin{equation}\label{eqn:noncollapsing}
\vol_{g(-\tau_l)}( B_{g(-\tau_l)} (x, r) ) \geq \kappa_0 r^n.
\end{equation}
Note that for every $r\leq \gamma^{l/2}$ 
\begin{equation*}
\frac{B}{\gamma^l \alpha}  \leq \frac{ B}{r^2 \alpha}=\frac{1}{(\zeta r)^{2}},
\end{equation*}
where $\zeta=\big(\frac{\alpha}{ B}\big)^{1/2}\leq 1$, since we can assume without loss of generality that $B>1$. 

Thus, we may now use \eqref{eqn:noncollapsing} to estimate, for every $r\leq \gamma^{l/2}$,
\begin{equation*}
\vol_{g(-\tau_l)}(B_{g(-\tau_l)}(x, r)) \geq \vol_{g(-\tau_l)}(B_{g(-\tau_l)}(x, \zeta r))\geq \kappa_0 \zeta^n r^n.
\end{equation*}
The lower bound of the claim now follows by putting $\kappa_2=\zeta^n\kappa_0$.

The upper bound directly follows from the $\kappa_1$ non-inflating property since $\gamma^{l/2}\leq 1$.
\end{proof}

\begin{lemma}[Covering lemma]\label{lem:covering}
There are $\alpha(B,\gamma), \delta(B,\gamma,\kappa_0,\eta)>0$ so that the construction of Sections \ref{subsec:energy_decomp} and \ref{subsec:quant_diff} satisfies the following: there exist $C_1,C_2<+\infty$  and $\beta(B,\gamma) \in(0,1/2)$ such that, for every $x\in M$, any $\mathbf a \in \{0,1\}^j$, $j\geq 1$, the set $S^k_{\eta,\gamma^{j-1}\alpha}\cap E_{\mathbf a}\cap B_{g(-\alpha)}(x,\beta)$ is covered by at most
$$C_1 (C_2\gamma^{-k})^j$$
$g(-\tau_{j-1})$ metric balls of radius $r_{j-1}$ centered at $S^k_{\eta,\gamma^{j-1}\alpha}$, where $\tau_j=\gamma^j \alpha$ and $r_j=\gamma^{j/2}\beta$.

In particular, $C_1$ depends only on $n,\kappa_0, \kappa_1,B,\gamma$ and $C_2$ only on $n$.
\end{lemma}
\begin{proof}
We prove this by induction. For $j=1$ we only need to estimate the number $P$ of balls $B_{g(-\tau_0)}(y_i, r_0)$, $i=1,\ldots, P$, in a minimal  covering of 
$$S^k_{\eta,\alpha}\cap B_{g(-\alpha)}(x,\beta),$$
where $y_i\in S^k_{\eta,\alpha} \cap B_{g(-\alpha)}(x,\beta)$. Note that $\alpha,\beta$ will be chosen later, depending only on $B$ and $\gamma$.

This number is bounded above by the cardinality $Q$ of a maximal $\beta$-separated at time $t=-\alpha$, subset $\{y_1,\ldots,y_Q\}$ of  $S^k_{\eta,\alpha} \cap B_{g(-\alpha)}(x,\beta)$. 

If $\beta\leq\frac{1}{2}$ we can apply Lemma \ref{lem:vol_control} to estimate
\begin{align*}
Q \kappa_2 (\beta/2)^n&\leq \sum_{i=1}^{Q} \vol_{g(-\alpha)}(B_{g(-\alpha)}(y_i , \beta/ 2)) \\
& \leq \vol_{g(-\alpha)}( B_{g(-\alpha)} (x, 2r  )) \leq \kappa_1 (2 \beta)^n,
\end{align*}
hence $P\leq Q\leq c_0:=\frac{\kappa_1 4^n}{\kappa_2}$.

We proceed to the induction step. Given any $\mathbf a\in \{0,1\}^{j+1}$  denote by $\tilde{\mathbf a}\in \{0,1\}^j$ the vector with $\tilde a_l =a_l$ for every $1\leq l\leq j$.

Now, recall that $\tau_j=\gamma^j \alpha$, $r_j=\gamma^{j/2} \beta$  and  suppose that 
$$S^k_{\eta,\tau_{j-1}}\cap E_{\tilde{\mathbf a}}\cap B_{g(-\alpha)}(x, \beta) \subset \bigcup_{i=1}^N B_{g(-\tau_{j-1})}(z_i, r_{j-1}),$$
 where $z_i\in S^k_{\eta,\tau_{j-1}}\cap E_{\tilde{\mathbf a}}\cap B_{g(-\alpha)}(x, \beta)$.
 
 First, observe that the curvature bound of the class $\mathcal C(n,B,\kappa_0,\kappa_1)$ implies that 
 $$|\riem(g(-\tau))|_{g(-\tau)} \leq \frac{B}{\gamma^j\alpha},$$
 in $M$ for $\tau \in [\tau_j ,\tau_{j-1}]$. Then, standard distance distortion estimates imply that for every $y\in M$ and $r>0$, 
 \begin{equation*}
 B_{g(-\tau_j)}(y,A^{-1}r) \subset B_{g(-\tau_{j-1})}(y,r) \subset B_{g(-\tau_j)}(y,Ar),
 \end{equation*}
 where $A= e^{c(n)B\gamma^{-1}}$.
 
 Thus, each ball of the given cover satisfies 
 $$B_{g(-\tau_{j-1})}(z_i, r_{j-1})\subset  B_{g(-\tau_j)}(z_i, Ar_{j-1}).$$
It follows that the cardinality $L$ of a maximal $r_j$-separated at time $t=-\tau_j$ set $\{w_1,\ldots, w_L\}$ in 
$$S^k_{\eta,\tau_j}\cap E_{\mathbf a} \cap B_{g(-\tau_{j-1})}(z_{i_0}, r_{j-1})\subset S^k_{\eta,\tau_j}\cap E_{\mathbf a} \cap B_{g(-\tau_j)}(z_{i_0}, A r_{j-1})$$ 
can be estimated by
 \begin{align*}
 L \kappa_2 (r_j /2)^n &\leq \sum_{i=1}^L \vol_{g(-\tau_j)}(B_{g(-\tau_j)}(w_i , r_j /2  ) ),\\
 &\leq \vol_{g(-\tau_j)} ( B_{g(-\tau_j)} ( z_{i_0}, A r_{j-1} + r_{j-1}  )   ),\\
 &\leq \kappa_1 (A+1)^n (r_{j-1})^n,
 \end{align*}
where we used again Lemma \ref{lem:vol_control}, assuming that $\beta\leq (1+A)^{-1}$. This provides us with an estimate $L\leq \frac{\kappa_1}{\kappa_2} (A+1)^n 2^n \gamma^{-n}=:c_1 $, $c_1=c_1(n,\alpha,B,\kappa_0,\kappa_1,\gamma)$.

 Thus the set
$$S^k_{\eta,\tau_j} \cap E_{\mathbf a}  \cap B_{g(-\tau_{j-1})}(z_{i_0}, r_{j-1})$$ 
can be covered by at most $c_1 $ balls $B_{g(-\tau_j)}(w_i,r_j)$, with centers in $S^k_{\eta,\tau_j} \cap E_{\mathbf a} \cap B_{g(-\tau_{j-1})}(z_{i_0}, r_{j-1})$.

 At this point it is clear that for the arguments above to go through we need to choose $\beta(B,\gamma)=\min\{1/2,(1+A)^{-1}\}$.
 
The rough estimate above is valid on all scales, and relies on the Type I assumption. On the other hand, if we are on a `good' scale $\tau_{j-1}$, namely a scale on which the flow looks selfsimilar, we can do much better.  

To see this, suppose that $a_j=0$ and let $B_{g(-\tau_{j-1})}(z_{i_0}, r_{j-1})$ be one of the balls in the cover of $S^k_{\eta,\tau_{j-1}}\cap E_{\tilde{\mathbf a}} \cap B_{g(-\alpha)}(x,\beta)$.

We will show that there is a minimal cover of 
$$S^k_{\eta,\tau_j} \cap E_{\mathbf a}\cap B_{g(-\tau_{j-1})}(z_{i_0}, r_{j-1})$$
 by at most $c_2(n) \gamma^{-k}$ balls $B_{g(-\tau_j)} (w_i , r_j)$ with centers $w_i\in S^k_{\eta,\tau_j} \cap E_{\mathbf a}$.

First, observe that  $E_{\mathbf a}\subset E_{\tilde{\mathbf a}}\subset L_{\tau_{j-1},\delta}$, since $a_j=0$, and recall that $z_{i_0} \in E_{\tilde{\mathbf a}}$. 

Chose $\delta=\delta_3(B, \gamma, \kappa_0, \beta/\alpha,\mu,\mu)$ as given by Lemma \ref{lem:line_up}, where $\alpha,\mu$ will be chosen later. 

Lemma \ref{lem:line_up} then implies that
\begin{equation}\label{eq:in_tube}
\begin{aligned}
&E_{\mathbf a} \cap B_{g(-\tau_{j-1})}(z_{i_0}, r_{j-1}) \subset   L_{\tau_{j-1},\delta}\cap B_{g(-\tau_{j-1})}(z_{i_0}, r_{j-1}),\\
&\subset T^{g(-\tau)}_{(D+\mu)\sqrt\tau}(\exp_{g(-\tau_j), z_{i_0}}(V\cap B_0(2r_{j-1})) ) \cap B_{g(-\tau_{j-1})}(z_{i_0}, r_{j-1})&,
\end{aligned}
\end{equation}
for every $\tau \in [\mu \tau_{j-1}, (2-\mu)\tau_{j-1}]$ and some $l$-dimensional subspace $V^l$ of $T_{z_{i_0}} M$. Moreover, $\mathfrak g_{z_{i_0}}$ is $(\mu,\tau_{j-1}^{1/2}, l,B)$-selfsimilar with respect to $V$.

Now, chose $\mu=\min\{\eta,\gamma,D\}$. On the one hand, this choice ensures that $l\leq k$, since $z_{i_0}\in S^k_{\eta, \tau_{j-1}}$. It also ensures that $\tau_j\in [\mu \tau_{j-1}, (2-\mu)\tau_{j-1}]$, thus  estimate \eqref{eq:in_tube} holds for  $\tau=\tau_j$.

Finally, chose $\alpha$ small enough so that $2D \sqrt{\alpha}<\beta$, so that 
\begin{equation}\label{eqn:tube_ineq}
(D+\mu)\sqrt{\tau_j}<r_j/10.
\end{equation}
This implies that there exists $C_2(n)$ and a minimal cover of 
$$S^k_{\eta,\tau_j}\cap E_{\mathbf a} \cap  B_{g(-\tau_{j-1})}(z_{i_0}, r_{j-1}) $$
 with at most $C_2\gamma^{-k}$ balls at time $t=-\tau_j$ of radius $r_j$ centered at $S^k_{\eta,\tau_j}$. 
 
 To construct such cover, first consider a maximal $r_j/4$-separated set in $\exp_{g(-\tau_j),z_{i_0}} ( V\cap B_0(r_{j-1}))$. Then, by \eqref{eqn:tube_ineq}, the $g(-\tau_j)$-balls of radius $r_j/2$ with centers in that set cover $S^k_{\eta,\tau_j}\cap E_{\mathbf a} \cap  B_{g(-\tau_{j-1})}(z_{i_0}, r_{j-1}) $, for $\mu$ small enough (but independent of the  other parameters of the proof). Finally, we can substitute each ball in this cover, with a ball of radius $r_j$ centered at $S^k_{\eta,\tau_j}\cap E_{\mathbf a} \cap  B_{g(-\tau_{j-1})}(z_{i_0}, r_{j-1}) $.
 
 Since there are at most $K$ `bad' scales and for the remaining $j-K$ we have the above more refined covering estimate, we obtain the result setting $C_1=c_0 c_1^K C_2^{-K}\gamma^{-K}$.
\end{proof}

\subsection{Proof of Theorem \ref{thm:vol_estimate}}
Given $B<+\infty$ and $\gamma$ which will be appropriately chosen later, let $\alpha, \beta$ be given by Lemma \ref{lem:covering}. 

It suffices to prove the theorem for $ \tau=\tau_{j-1}$ for all $j\geq 1$,  since for any  $\tau_j< \tau < \tau_{j-1}$, 
\begin{align*}
\vol_{g(- \tau)}\left(S^k_{\eta, \tau}\cap B_{g(-\alpha)}(x,\beta) \right)&\leq \vol_{g(- \tau)}\left(S^k_{\eta, \tau_{j-1}}\cap B_{g(-\alpha)}(x,\beta)\right),\\
&\leq C_\eta  \tau_{j-1}^{\frac{n-k-\eta}{2}},\\
&\leq C_\eta (\gamma^{-1})^{\frac{n-k-\eta}{2}} \tau^{\frac{n-k-\eta}{2}}.
\end{align*}
Now, recall that $M= \bigcup_{\mathbf a\in \{0,1\}^j} E_{\mathbf a}$ and that there are at most $2j^K$ non-empty sets $E_{\mathbf a}$. Moreover, from Lemma \ref{lem:covering}, $S^k_{\eta,\tau_{j-1}} \cap E_{\mathbf a} \cap B_{g(-\alpha)}(x,\beta)$ is covered by at most  $C_1 (C_2 \gamma^{-k})^{j}$ balls at time $t=-\tau_{j-1}$ of radius $r_{j-1}$. Thus, using Lemma \ref{lem:vol_control}:
\begin{align*}
&\vol_{g(-\tau_{j-1})}\left( S^k_{\eta,\tau_{j-1}} \cap B_{g(-\alpha)}(x,\beta) \right) \\
&\leq 2 j^K C_1 (C_2 \gamma^{-k})^{j} \kappa_1 (2 r_{j-1})^n.
\end{align*}
Now, we chose $\gamma=\gamma(n,\eta)$ small enough so that $C_2 \leq \gamma^{-\eta/2}$ and we can also bound $j^K \leq C(K,\eta,\gamma) (\gamma^{j-1})^{-\eta/2}$. The estimate above then becomes
\begin{align*}
\vol_{g(-\tau_{j-1})}\left( S^k_{\eta,\tau_{j-1}} \cap B_{g(-\alpha)}(x,\beta) \right) &\leq  C_\eta \tau_{j-1}^{\frac{n-k-\eta}{2}},
\end{align*}
which is what we want to prove. \qed

\begin{remark}\label{rmk:later_est}
Note that due to the standard lower scalar curvature bound for the Ricci flow $R(g(-\tau))\geq -\frac{n}{2(\tau+2 )}$ and the evolution of the volume under Ricci flow, for every $0<\bar \tau\leq \tau\leq \alpha$
\begin{align*}
\vol_{g(-\bar \tau)}\left(S^k_{\eta,\tau} \cap B_{g(-\alpha)}(x,\beta)  \right) &\leq  c(n) \vol_{g(-\tau)} \left( S^k_{\eta,\tau} \cap B_{g(-\alpha)}(x,\beta) \right ) \\
&\leq  C_\eta  \tau^{\frac{n-k-\eta}{2}}.
\end{align*}
Moreover, if $\Omega=\{ x\in M,\; \sup_{t\in[0,T)}|\riem(g)|_{g}(x,t)<+\infty\}$, then
\begin{equation}
\vol_{g(0)}(S^k_{\eta,\tau} \cap B_{g(-\alpha)} (x,\beta) \cap \Omega) \leq  C_\eta  \tau^{\frac{n-k-\eta}{2}}.
\end{equation}
\end{remark}

\section{Curvature estimates}\label{section:c_estimates}
Let $(M,g(t))_{t\in (-2,0)}$ be a complete Ricci flow satisfying 
\begin{equation}\label{type}
\max_M |\riem(g(t))|_{g(t)}\leq \frac{B}{|t|}.
\end{equation}
for all $t\in (-2,0)$.
 If $(g(t))_{t\in (-2,0)}$ is not singular at $x\in M$, namely there is a neighbourhood $U$ of $x$ such that  
 $$\sup_{U\times (-2,0)} |\riem(g(t))|_{g(t)}<+\infty,$$
we can define the curvature radius at $x$ as
\begin{equation*}
r_{\riem}(x)=\sup\left\{r\leq 1, \; |\riem(g)|\leq r^{-2}\;\textrm{in}\; B_{g(-r^2)}(x,r)\times [-r^2,0] \right\}.
\end{equation*}
If $(g(t))_{t\in (-2,0)}$ is singular at $x$, we define $r_{\riem}(x)=0$.
\subsection{$\varepsilon$-regularity} Below we prove a few $\varepsilon$-regularity results for Ricci flows satisfying \eqref{type}, which imply that high curvature regions are inside one of the sets $S^k_{\varepsilon,\tau}$.

\begin{lemma}[$\varepsilon$-regularity]\label{lem:epsilon}
For every $B<+\infty$ and $\kappa>0$, there exists $\varepsilon(B,\kappa)>0$ such that if a complete Ricci flow $(M,g(t))_{t\in (-2,0)}$ satisfies \eqref{type} and is $\kappa$ non-collapsed below scale $1$, then for every $\tau\in (0,1]$
\begin{equation*}
\{ r_{\riem} < \sqrt\tau \}   \subset   S^{n-2}_{\varepsilon, \tau}.
\end{equation*}
Moreover, if $\dim M=4$ and $g(t)$ has positive isotropic curvature, then for every $\tau\in (0,1]$
\begin{equation*}
\{ r_{\riem} < \sqrt\tau \} \subset  S^1_{\varepsilon,\tau}.
\end{equation*}
\end{lemma}
\begin{proof}
To prove the first statement, take a sequence of counterexamples $(M_i,g_i(t))_{t\in (-2,0)}$ satisfying \eqref{type}, and $x_i\in M_i$, $\tau_i \in (0,1]$, $\varepsilon_i\searrow 0$  such that $r_{\riem}(x_i) < \sqrt{\tau_i}$ and $x_i\not\in S^{n-2}_{\varepsilon_i, \tau_i}$.

Thus, the pointed flows $\mathfrak g_i=(M_i, g_i (t), x_i)_{t\in (-2,0)}$ are $(\varepsilon_i, s_i, n-1,B)$-selfsimilar, for some $s_i\in [\tau_i^{1/2},1]$. By Lemma \ref{lem:soliton_compactness}, a subsequence of $(\mathfrak g_i)_{s_i}$ converges to a shrinking Ricci soliton that splits at least $n-1$ Euclidean factors. The only such soliton is the Gaussian shrinking soliton. By Perelman's pseudolocality theorem  \cite{Perelman02} we conclude that $r_{\riem}(x_i) \geq s_i \geq \tau_i^{1/2}$ for large $i$, which is a contradiction.

The proof of the second statement is similar, with the difference that the limiting soliton now splits at least two Euclidean factors and has positive isotropic curvature. However, four dimensional gradient shrinking Ricci solitons with positive isotropic curvature split at most one Euclidean factor, by \cite{LiNiWang16}, which is a contradiction.
\end{proof}

Under an additional bound on the Weyl curvature $W$, we can improve Lemma \ref{lem:epsilon} as follows.

\begin{lemma}[$\varepsilon$-regularity under Weyl curvature bound]\label{lem:epsilon_weyl}
Given $B<+\infty$ and $\kappa>0$, there exists  $\varepsilon(B,\kappa)>0$ such that if  for some $x\in M$ and $0<r\leq 1$ a complete Ricci flow $(M,g(t))_{t\in (-2,0)}$ satisfies \eqref{type}, it is $\kappa$ non-collapsed below scale $1$, and
\begin{enumerate}
\item $(M,g(t),  x)_{t\in (-2,0)}$ is $(\varepsilon, r, 2,B)$-selfsimilar,
\item $r^2 |W(g(-r^2))|_{g(-r^2)}<\varepsilon$ in $B_{g(-r^2)}( x, \varepsilon^{-1} r)$,
\end{enumerate}
 then $r_{\riem} (x) \geq  r$. 
\end{lemma}
\begin{proof}
We argue by contradiction. Let $(M_i,g_i(t))_{t\in (-2,0)}$ be a sequence satisfying \eqref{type}, $x_i\in M_i$  and suppose that there are sequences $r_i\in(0,1]$ and $\varepsilon_i\searrow 0$  such that 
\begin{equation}\label{weyl_zero}
 r_i^2 |W(g_i(-r_i^2))|_{g_i(-r_i^2)}<\varepsilon_i
\end{equation}
in $B_{g_i(-r_i^2)}(x_i, \varepsilon_i^{-1} r_i)$ and $(M_i,g_i(t), x_i)_{t\in (-2,0)}$ is $(\varepsilon_i, r_i, 2,B)$-selfsimilar, but $r_{\riem}(x_i)< r_i$. 

By Lemma \ref{lem:soliton_compactness} there is a subsequence of $(M_i,r_i^{-2}g_i(r_i^2 t), x_i)_{t\in (-2,0)}$ converging to a shrinking Ricci soliton $(N, h(t), q)_{t\in (-2,0)}$, which splits at least 2 Euclidean factors.

 Inequality \eqref{weyl_zero} implies that $(N,h(t))$ has vanishing Weyl curvature. Since it splits more than one Euclidean factor, it has to be the Gaussian shrinking soliton, by \cite{Zhang09}. Perelman's pseudolocality theorem \cite{Perelman02}  then gives that  $r_{\riem}( x_i)\geq r_i$, which is a contradiction.
\end{proof}

\subsection{Regularity estimates}
We now couple the $\varepsilon$-regularity results of Lemmata \ref{lem:epsilon} and \ref{lem:epsilon_weyl} with the volume estimate of Theorem \ref{thm:vol_estimate} to prove the following.
 
 \begin{theorem}\label{thm:regularity1}
 Given $(M,g(t))_{t\in (-2,0)}\in \mathcal C(n,B,\kappa_0,\kappa_1)$ and $\eta\in (0,1)$ there exist $\alpha(B), \beta(B)>0$ and $C_{\eta}=C(n,B, \kappa_0,\kappa_1,\eta)<+\infty$ such that for every $x\in M$ and $0<\tau \leq \alpha$
 \begin{align}
 \vol_{g(0)}\left(  \{0<r_{\riem}<\sqrt\tau\}   \cap B_{g(-\alpha)}(x, \beta) \} \right) &\leq C_{\eta}  \tau^{1-\eta},\label{vol1}\\
\textrm{and}\qquad\int_{B_{g(-\alpha)}(x,\beta)\cap\{r_{\riem}>0\} } r_{\riem}^{-2(1-\eta)} d\mu_{g(0)} &\leq C_{\eta}.\label{radius1}
 \end{align}

 If in addition $\dim M=4$ and $g(t)$ has positive isotropic curvature then 
 \begin{align}
 \vol_{g(0)}\left(\{0<r_{\riem}<\sqrt\tau\} \cap B_{g(-\alpha)}(x, \beta)  \right) &\leq C_{\eta}  \tau^{\frac{3}{2}-\eta},\label{vol2}\\
\textrm{and}\qquad \int_{B_{g(-\alpha)}(x,\beta)\cap \{ r_{\riem}>0 \}} r_{\riem}^{-3(1-\eta)} d\mu_{g(0)} &\leq C_{\eta}.\label{radius2}
 \end{align}
 \end{theorem}
 
 \begin{proof}
 Let $\alpha(B)$ and $\beta(B)$ be given by Theorem \ref{thm:vol_estimate}. Then, estimates \eqref{vol1} and \eqref{vol2} easily follow from the volume estimate of Theorem \ref{thm:vol_estimate}, Remark \ref{rmk:later_est} and Lemma \ref{lem:epsilon}.  
 
To prove \eqref{radius1} and \eqref{radius2} we compute
 \begin{align*}
 &\int_{B_{g(-\alpha)}(x,\beta)\cap \{r_{\riem}>0\}} r_{\riem}^{-p} d\mu_{g(0)}  =\\
 & =\int_{B_{g(-\alpha)}(x,\beta)\cap  \{r_{\riem}>0\} } \left(  \frac{1}{p}\int_{r_{\riem}}^1 s^{-(p+1)} ds + 1 \right) d\mu_{g(0)}, \\
&\leq  \frac{1}{p} \int_0^1 \frac{1}{s^{p+1}} \vol_{g(0)}\big(  \{0< r_{\riem} \leq  s \} \cap B_{g(-\alpha)}(x,\beta)\big) ds +\vol_{g(0)}(B_{g(-\alpha)}(x,\beta)),\\
&\leq C(\eta,p,n, B,\kappa_0,\kappa_1) \int_0^1 s^{-(p+1)+l-\eta}ds +\vol_{g(0)}(B_{g(-\alpha)}(x,\beta)).
\end{align*}
For the last inequality we used either \eqref{vol1} or \eqref{vol2}, substituting $l=2$ or $l=3$ respectively. Moreover, $\vol_{g(0)}(B_{g(-\alpha)}(x,\beta))$ should be interpreted as $\vol_{g(0)}(B_{g(-\alpha)}(x,\beta) \cap \{r_{\riem}>0\})$.

Thus, for every $p= l -2\eta$ we can  bound, for some $C_p=C(p,n,B,\kappa_0,\kappa_1)$,
\begin{align*}
\int_{B_{g(-\alpha)}(x,\beta)\cap \{r_{\riem}>0\}} r_{\riem}^{-p} d\mu_{g(0)} &\leq C_p +\vol_{g(0)}(B_{g(-\alpha)}(x,\beta)),\\
& \leq C_p+ C(n) \vol_{g(-\alpha)}(B_{g(-\alpha)}(x,\beta)),\\
& \leq C_p + C(n,\kappa_1) \beta^n.
\end{align*}
Here, we used the volume control due to the standard scalar curvature bound $R\geq -\frac{n}{2(\tau+2)}$, as in Remark \ref{rmk:later_est}, and Lemma \ref{lem:vol_control}.  This suffices to prove \eqref{radius1} and \eqref{radius2}.
 \end{proof}

 \begin{theorem}\label{thm:regularity2}
 Given $(M,g(t))_{t\in(-2,0)}\in\mathcal C(n,B,\kappa_0,\kappa_1)$ and $\eta\in(0,1)$ there exist $\alpha(B), \beta(B)>0$, $\varepsilon(B)>0$ and $C_\eta=C(n,B, \kappa_0, \kappa_1,\eta)<+\infty$ such that if for every $t\in(-2,0)$
 \begin{equation}
 \sup_{B_{g(-\alpha)}(x, 2\varepsilon^{-1} \beta) } |W(g(t))|_{g(t)} <\varepsilon,\label{weyl_small}
 \end{equation}
 then  for every $0<\tau \leq\alpha$
 \begin{align}
 \vol_{g(0)}\left(\{0<r_{\riem}<\sqrt\tau\}\cap B_{g(-\alpha)}(x, \beta)  \right) &\leq C_{\eta}  \tau^{\frac{n-1}{2}-\eta},\\ 
 \textrm{and}\qquad \int_{B_{g(-\alpha)}(x,\beta) \cap \{r_{\riem}>0\}} r_{\riem}^{-(n-1)(1-\eta)} d\mu_{g(0)} &\leq C_{\eta}. \label{radius3}
 \end{align}
 \end{theorem}
 \begin{proof}
 Let $\alpha,\beta$ given by Theorem \ref{thm:vol_estimate} and $\varepsilon$ by Lemma \ref{lem:epsilon_weyl}. Also, recall the following estimate from Lemma 2.6 of \cite{Naber10}: along any unit speed minimizing  $g(t)$-geodesic $\sigma(s)$, $s\in [0, l]$, 
\begin{equation}
\int_0^l \ric(\dot \sigma(s), \dot\sigma(s)) ds \leq \frac{C_1}{\sqrt{|t|}}, 
\end{equation}
for some constant $C_1=C_1(n,B)<+\infty$. It follows that
\begin{equation}\label{d_distance}
\frac{d}{dt} d_{g(t)} (y,z) \geq - \frac{C_2(n,B)}{ \sqrt{|t|}}.
\end{equation}
Integrating \eqref{d_distance} gives, for every $y\in B_{g(-\alpha)}(x,\beta)$ and $t\in [-\alpha,0)$,
\begin{equation}
B_{g(t)}(y, \varepsilon^{-1} \beta) \subset B_{g(-\alpha)}(x, \beta(1 + \varepsilon^{-1}) + C_3 \sqrt\alpha),
\end{equation}
  where $C_3=C_3(n,B)<+\infty$.
 
 Choosing $\varepsilon>0$ small enough so that
 \begin{equation*}
 B_{g(-\alpha)}(x,\beta(1+\varepsilon^{-1}) + C_3 \sqrt\alpha) \subset B_{g(-\alpha)}(x,2\varepsilon^{-1} \beta),
 \end{equation*}
 and using Lemma \ref{lem:epsilon_weyl}, we obtain that for every $r\in(0,2\beta]$ 
 \begin{equation}
 \{r_{\riem}<r\}\subset S^1_{\epsilon,r^2}.
 \end{equation}
 Note that $2\beta\leq 1$ by the proof of Theorem \ref{thm:vol_estimate}. The result then follows by arguing as in Theorem \ref{thm:regularity1}.
 \end{proof}
 
  \begin{proof}[Proof of Theorem \ref{introthm}]
 Estimate \eqref{curvature1} is an immediate consequence of estimates \eqref{radius1} and \eqref{radius2}, since Shi's local derivative estimates imply
\begin{align*}
&\int_{B_{g(-\alpha)}(x,\beta)\cap \{r_{\riem}>0\} } |\nabla^j \riem(g(0))|_{g(0)}^p d\mu_{g(0)} \\
&\leq C(n,p,j)\int_{B_{g(-\alpha)}(x,\beta) \cap \{r_{\riem}>0\} } r_{\riem}^{-(j+2)p} d\mu_{g(0)}. 
\end{align*}
 \end{proof}
 
 \begin{remark}\label{weyl1}
Under the assumptions of Theorem \ref{introthm}, if in addition the Weyl curvature satisfies assumption \eqref{weyl_small} of Theorem \ref{thm:regularity2}, then the estimates of Theorem \ref{introthm} hold for any $p\in (0,n-1)$.
\end{remark}

 \subsubsection{General Type I Ricci flows.} Given any complete Ricci flow $(M,g(t))_{t\in [0,T)}$, $T>1$, we may define the curvature radius of $g(t)$ at a non-singular point $(x,t)\in M\times [1,T]$ as
 \begin{equation*}
r_{\riem}(x,t)=\sup\left\{r\leq 1, \; |\riem(g)|\leq r^{-2}\;\textrm{in}\; B_{g(t-r^2)}(x,r)\times [t-r^2,t] \right\},
\end{equation*}
and $r_{\riem}(x,T)=0$, if $(x,T)$ is singular.
 
 The following theorem holds:
 
\begin{theorem}\label{thm:integral_estimates}
Let $(M^n,g(t))_{t\in [0,T)}$, $\dim M=n$ and $T>1$, be a compact Ricci flow satisfying \eqref{type} for some constant $B<+\infty$. Then for every \linebreak$p\in (0,2)$, there exists $C_p=C(g(0),p)<+\infty$ such that
\begin{equation}\label{eqn:int_estimate}
\int_{M\cap \{r_{\riem}(\cdot,t)>0\}} r_{\riem}^{-p} (\cdot,t) d\mu_{g(t)} \leq C_p
\end{equation}
for every $t\in [1,T]$.

Moreover,  if $\dim M=4$ and $g(t)$ has positive isotropic curvature, or if $\sup_{M\times [0,T)}|W(g(t))|_{g(t)} < +\infty$, the estimate above holds for $p\in (0,n-1)$.
\end{theorem}
\begin{proof}
First, observe that, due to the non-collapsing \cite{Perelman02} and non-inflating \cite{Zhang12} properties of the Ricci flow,  there exist $\kappa_0,\kappa_1>0$ and $\rho>0$, which depend on $g(0), T$ and $B$, such that the following holds: for every $\bar t\in [1,T]$ the flow $(M,\rho^{-2}g(\rho^2 t+\bar t))_{t\in(-2,0)}$ is in the class $\mathcal C(n,B,\kappa_0,\kappa_1)$.

Now, let $\alpha,\beta$ be provided by applying Theorem \ref{thm:vol_estimate} to the class $\mathcal C(n,B,\kappa_0,\kappa_1)$. Moreover, let $N(t)$ be the minimal number of $g(t)$-balls of radius $\rho\beta$ required to cover $M$. 

For any $p\in(0,2)$, applying Theorem \ref{thm:regularity1} to $(M,\rho^{-2}g(\rho^2 t+\bar t))_{t\in(-2,0)}$ gives
\begin{align*}\label{eqn:main_thm_1}
&\int_{M\cap \{r_{\riem}(\cdot,\bar t)>0\}} r_{\riem}^{-p} (\cdot,\bar t) d\mu_{g(\bar t)} \leq\\
&\leq \sum_{i=1}^{N(\bar t-\rho^2\alpha)} \int_{B_{g(\bar t-\rho^2\alpha)}(x_i,\rho\beta)\cap \{r_{\riem}(\cdot,\bar t)>0\}} r_{\riem}^{-p} (\cdot,\bar t) d\mu_{g(\bar t)},\\
&\leq N(\bar t-\rho^2 \alpha) C(n,p,B,\kappa_0,\kappa_1) \rho^{\frac{n}{2}-p}.
\end{align*}
To conclude the proof, note that we can estimate $N(\bar t-\rho^2\alpha)\leq C(g(0),B)$, since $|\riem(g(t))|_{g(t)}\leq \frac{B}{\rho^2\alpha}$ for $t\leq T-\rho^2\alpha$.

The remaining assertions of the theorem follow from a similar line of reasoning, applying Theorems \ref{thm:regularity1} and \ref{thm:regularity2} respectively. Note that, in order to apply Theorem \ref{thm:regularity2} when there is a uniform bound on the Weyl curvature, we need to chose $\rho>0$ small enough so that $\rho^2 \sup_{M\times [0,T)}|W(g(t))|_{g(t)} <\varepsilon$, $\varepsilon$ given by Theorem  \ref{thm:regularity2}.
\end{proof}

\begin{proof}[Proof of Theorem \ref{thm:curv_estimates}]
Estimate \eqref{eqn:introder_curv_estimate1} follows from Theorem \ref{thm:integral_estimates}, as in the proof of Theorem  \ref{introthm}. To obtain estimate \eqref{eqn:introder_curv_estimate2} we first write $r_{\riem}^{-(j+2)p}=r_{\riem}^{-\frac{l(j+2)p}{l+2}}r_{\riem}^{-\frac{2(j+2)p}{l+2}}$, substituting $l=2$ or $3$, depending on whether we are in the general case or the case of positive isotropic curvature respectively.

Then we estimate
\begin{align*}
&\int_{1}^T \int_{M\cap \{r_{\riem}(\cdot,s)>0\}}  |\nabla^j \riem(g(s))|_{g(t)}^{p} d\mu_{g(s)} ds \leq \\
&\leq C(n,p,j) \int_{1}^T \int_{M\cap \{r_{\riem}(\cdot,s)>0\}} r_{\riem}^{-(j+2)p} d\mu_{g(s)} ds, \\
&=  C(n,p,j) \int_{1}^T \int_{M\cap \{r_{\riem}(\cdot,s)>0\}} r_{\riem}^{-\frac{l(j+2)p}{l+2}}r_{\riem}^{-\frac{2(j+2)p}{l+2}} d\mu_{g(s)} ds, \\
&\leq C(n,p,j) \int_{1}^T \int_{M\cap \{r_{\riem}(\cdot,s)>0\}} r_{\riem}^{-\frac{l(j+2)p}{l+2}} \left(B /|s| \right)^{\frac{(j+2)p}{l+2}} d\mu_{g(s)} ds, 
\end{align*}
which implies the required bound, as long as $p\in (0,\frac{l+2}{j+2})$, by Theorem \ref{thm:integral_estimates}.
\end{proof}
\begin{remark}\label{weyl2}
Under the assumptions of Theorem \ref{thm:integral_estimates}, if the Weyl curvature is uniformly bounded for all $t\in [0,T)$, then the estimates of Theorem \ref{thm:curv_estimates} hold for any $p\in (0,n-1)$.
\end{remark}

\providecommand{\bysame}{\leavevmode\hbox to3em{\hrulefill}\thinspace}
\providecommand{\MR}{\relax\ifhmode\unskip\space\fi MR }

\providecommand{\MRhref}[2]{%
  \href{http://www.ams.org/mathscinet-getitem?mr=#1}{#2}
}
\providecommand{\href}[2]{#2}

\end{document}